\documentclass[12pt]{amsart}
\usepackage{amsmath,amssymb,enumerate}
\usepackage{color}

\newtheorem{theorem}{Theorem}[section]
\newtheorem{theorem*}{Theorem}
\newtheorem{claim}{Claim}[theorem]

\newtheorem{lemma}[theorem]{Lemma}

\newtheorem{corollary}[theorem]{Corollary}

\newtheorem{remark}[theorem]{Remark}
\theoremstyle{definition}

\DeclareMathOperator{\si}{si}

\DeclareMathOperator{\cl}{cl}
\DeclareMathOperator{\PG}{PG}
\DeclareMathOperator{\AG}{AG}
\DeclareMathOperator{\GF}{GF}

\newcommand{\del}{\setminus}
\newcommand{\con}{/}

\begin{document}

\sloppy

\title{Dense $\PG(n-1,2)$-Free Binary Matroids}

\author{Rutger Campbell}
\email{}
\address{Department of Combinatorics and Optimization,
University of Waterloo, Canada}

\date{\today}

\begin{abstract}
For each integer $n \geq 2$, we prove that,
if $M$ is a simple rank-$r$ $\PG(n-1,2)$-free binary matroid with
$|M|>\left(1-\frac{3}{2^n}\right)2^r$, then there is a triangle-free corank-$(n-2)$ flat of $M$.
\end{abstract}

\maketitle

\section{Introduction}
We call a matroid {\em $N$-free}
if it has no restriction isomorphic to $N$.
A {\em triangle} is a matroid isomorphic to $\PG(1,2) \cong U_{2,3}$.
Note that if a binary matroid $M$ contains a triangle-free flat of corank-$(n-2)$,
then $M$ is $\PG(n-1,2)$-free.
We show that for sufficiently dense matroids, the converse holds.

\begin{theorem}\label{struct}
For integers $r$ and $n$ with $r \geq n \geq 2$, 			
if $M$ is a simple rank-$r$ $\PG(n-1,2)$-free binary matroid with
$|M|> \left(1-\frac{3}{2^{n}}\right)2^{r}$, 
then there is a triangle-free corank-$(n-2)$ flat $K$ of $M$.
Furthermore, $|K|> \frac{1}{4}2^{r(K)}$.
\end{theorem}

This theorem is tight for $r=4$ and $n=3$ because of $M(K_5)$ and tight
for all integers $r\geq 4$ and $n\geq 3$ with $r\geq n$  because
of constructions based on $M(K_5)$.

Bruen and Wehlau,
building upon results by Davydov and Tombak~[\ref{dt}],
give a precise description of all simple
triangle-free binary matroids with density at least $\frac{1}{4}$
(see~[\ref{bw}]).
By applying this description to $K$ in Theorem~\ref{struct},
we get a precise description of all simple $\PG(n-1,2)$-free binary matroids
with density at least $1-\frac{3}{2^n}$.
As a consequence of these results, 
we get Corollary~\ref{max_size} and Corollary~\ref{tidor_thm}.

\begin{corollary}\label{max_size}
For integers $r$ and $n$ with $r\geq n\geq 2$,
if $M$ is a maximal simple rank-$r$ triangle-free binary matroid with
$|M|> \left(1-\frac{3}{2^{n}}\right)2^{r}$,
then $|M| = \left(1-\frac{3}{2^{n}}\right)2^{r}+2^k$ for some integer $k \geq 1$.
\end{corollary}

Consider a simple rank-$r$ binary matroid $M$ 
as a restriction of a rank-$r$ projective geometry $G \cong \PG(r-1,2)$.
The {\em critical number}, $\chi(M)$, of $M$ is the minimum corank of 
the flats in $G$ that are disjoint from $M$.
Note that if $\chi(M)=n-1$, then $M$ is $\PG(n-1,2)$-free. 

Tidor~[\ref{tidor}] proved the following using Davydov and Tombak's~[\ref{dt}] result.
\begin{corollary}\label{tidor_thm}
For integers $r$ and $n$ with $r \geq n \geq 2$, 
if $M$ is a simple, rank-$r$ $\PG(n-1,2)$-free binary matroid with
$|M|> \left(1-\frac{3}{2^{n}}\right)2^{r}$, 
then $\chi(M)$ is $n$ or $n-1$.
\end{corollary}

In contrast, we have the following result due to Geelen and Nelson~[\ref{gn}].
\begin{corollary}
Fix some integer $n \geq 2$.
For each $\epsilon>0$ and each integer $c\geq n$,
there is a simple, $\PG(n-1,2)$-free, binary matroid $M$
such that $|M| \geq \left(1-\frac{3}{2^{n}}-\epsilon \right)2^{r}$ 
and $\chi(M)=c$.
\end{corollary}

Our proof of Theorem~\ref{struct} is by induction on $n$.

Our proof of this theorem
is inspired by Tidor's~[\ref{tidor}]
proof of Corollary~\ref{tidor_thm}, which is based the proof of Theorem~\ref{GS}
due to Goevaerts and Storme~[\ref{gs}].
They, in turn, get this inductive method from Beutelspacher~[\ref{b80}].

For the case $n=3$ we use a variation of Green's counting lemma~[\ref{green}]
from additive combinatorics.
Similar techniques are used in Tidor~[\ref{tidor}] and Geelen and Nelson~[\ref{gn}].

Our results are analogous to, and motivated by,
results on dense graphs with small clique number.
In particular, Theorem~\ref{struct} is the analogue of the following theorem due to Goddard and Lyle~[\ref{gl}].
\begin{theorem} \label{graphstruct}
For integers $k$ and $t$ with $k\geq t\geq 3$,
if $G$ is a simple $K_t$-free graph on $k$ vertices and with minimum degree
$\delta(G) > \frac{2t-5}{2t-3} k$,
then there is a partition of $G$ into a triangle-free
graph $H$ and a $(t-3)$-colourable graph $T$.
Furthermore, $H$ has minimum degree $\delta(H) > \frac{|V(H)|}{3}$.
\end{theorem}

This can be combined with Brandt and Thomass\'{e}'s~[\ref{bt}]
description of all triangle-free graphs $H$ with minimum degree
$\delta(H) > \frac{|V(H)|}{3}$ to give a complete characterization
of all simple $K_t$-free graphs on $k$ vertices with minimum degree greater than $\frac{2t-5}{2t-3} k$.

%
\section{Preliminaries}
\subsection{Geometry}
Simple $\PG(n-1,2)$-free binary matroids arise in various contexts~[\ref{packingproblem}].
Many previous results were originally solved in the context of geometry.
Let a {\em rank-$r$ binary representation} be an ordered pair $(E,G)$ 
comprised of an {\em ambient geometry} $G\cong \PG(r-1,2)$,
and a {\em ground set} $E$ that is a subset of the points of $G$.
We say that $(E,G)$ {\em represents} the simple binary matroid $M:=G|E$.
Note that the rank of $M$ is at most the rank of $(E,G)$, but equality need not hold.
We will say that $(E,G)$ is {\em $N$-free} when $M$ is $N$-free.
This geometric representation is convenient as it is easier to work with a concrete ambient geometry that we understand
as apposed to a general simple binary matroid.

Bose and Burton proved the following theorem geometrically~[\ref{bb}].
\begin{theorem}[Binary Bose-Burton]\label{binaryBB}
For integers $r$ and $n$ with $r \geq n \geq 2$,
if $(E,G)$ is a $\PG(n-1,2)$-free rank-$r$ binary representation,
then $|E| \leq \left(1-\frac{2}{2^{n}} \right)2^{r}$. Furthermore,
 if equality holds, then $E$ is contained in the complement of a corank-$n$ flat of $G$.
\end{theorem}
Goevaerts and Storme~[\ref{gs}] built upon this by proving the same 
conclusion for binary representations that have size close to the maximum.
\begin{theorem}[Goevaerts-Storme]\label{GS}
For integers $r$ and $n$ with $n \geq 2$ and $r \geq n+2$,
if $(E,G)$ is a $\PG(n-1,2)$-free rank-$r$ binary representation
 with $|E| > \left(1-\frac{2}{2^n}-\frac{3}{2^{n+2}}\right)2^{r}$,
then $E$ is contained in the complement of a corank-$n$ flat of $G$.
\end{theorem}

\subsection{Inductive lemmas}
These are the relevant lemmas that appear
in some analogous form in the papers using Beutelspacher's~[\ref{b80}]
inductive technique.


\begin{lemma}\label{Hsize}
For integers $r$ and $n$ with $r \geq n \geq 3$, 
let $(E,G)$ be a $\PG(n-1,2)$-free rank-$r$ binary representation of
a simple binary matroid $M$.
If $H$ is a hyperplane of $G$ such that
$(E \cap H,G|H)$ is not $\PG(n-2,2)$-free, then
$|E \del H| \leq \left(1-\frac{1}{2^{n-1}} \right) 2^{r-1}$.
Furthermore, if $|E| > \left( 1-\frac{3}{2^n}\right)2^r$,
then $|E \cap H|>\left(1-\frac{2}{2^{n-1}}\right)2^{r-1}$.
\end{lemma}
\begin{proof}
Suppose $H$ is a hyperplane of $G$ 
such that $E\cap H$ contains the points of some $S\cong \PG(n-2,2)$.
Consider the parallel classes of $(M \con S) \del (H \con S) \subseteq (G \con S) \del (H \con S)$.
Note $\si ((G \con S) \del (H \con S)) \cong \AG(r-n,2)$
so there are at most $2^{r-n}$ such sets.
Since $(E,G)$ is $\PG(n-1,2)$-free, each of
these sets must have size strictly less then $|\PG(n-1,2)|-|S|=|\AG(n-1,2)| =2^{n-1}$.

Thus as these parallel classes form a partition of the point in $M\del H$,
$$|E\del H|=|M \del H| \leq (2^{n-1}-1) 2^{r-n} = \left( 1- \frac{1}{2^{n-1}} \right) 2^{r-1}.$$

If we also assume that $|E| > \left( 1-\frac{3}{2^n}\right)2^r$, then
$$|E\cap H|=|E|-|E\del H| >\left(1-\frac{2}{2^{n-1}}\right)2^{r-1}.$$ 
\end{proof}

Let $(E,G)$ be a binary representation, and let $p \in E$.
Define the {\em cone in $E$ at $p$} to be $(E_p, G)$ where $E_p$ is the union of all lines in $(G|E) \del \{p\}$ that span $p$ in $G$.
Equivalently $E_p$ corresponds to all points in non-trivial parallel classes of $(G|E) \con \{p\}$.

\begin{lemma}\label{cone_ind}
For integers $r$ and $n$ with $r \geq n \geq 3$,								
let $(E,G)$ be a $\PG(n-1,2)$-free rank-$r$ binary representation.
Then for each $p \in E$ 
the cone $(E_p,G)$ is a $\PG(n-2,2)$-free rank-$r$ binary representation 
with $\left|E_p\right| \geq 2 |E| - 2^{r}$.
\end{lemma}
\begin{proof}
%
%
If $E_p$ contained some $S\cong \PG(n-2,2)$, then $E$ contains
$\cl_G(S\cup \{p\})\cong \PG(n-1,2)$.
Therefore $(E_p,G)$ is a $\PG(n-2,2)$-free rank-$r$ binary representation.

As $p$ is a point in $G\cong \PG(r-1,2)$, there are 
$|\PG(r-2,2)|=2^{r-1}-1$ lines in $G$ that pass through $p$.
Thus,
$|E\del \{p\}| \leq \frac{\left|E_p\right|}{2} +2^{r-1}-1$,
and so
$\left|E_p\right| \geq 2 |E| - 2^{r}$.

\end{proof}

\section{Fano-free binary representations}
For the case $n=3$, we will need a specialization 
of a result in additive combinatorics.
We begin with a definition from additive combinatorics.

Let $(E,G)$ be a rank-$r$ binary representation for some integer $r\geq 1$.
For any $\epsilon >0$,
we say that $(E,G)$ is 
{\em $\epsilon$-uniform} when
for any hyperplane $H$ of $G$ we have 
$$\frac{1}{2}\left(|E|-\epsilon2^r\right) \leq |E\cap H| \leq \frac{1}{2}\left(|E|+\epsilon2^r\right).$$
This is a qualitative description of how uniformly distributed $E$ is with regards to hyperplanes of $G$.
It is analogous to uniformity for graphs.

We now give a slight strengthening of a case of Green's
 ``counting lemma in $(\mathbb{Z} / 2\mathbb{Z})^{n}$"
(see Proposition 2.3 in~[\ref{green}]). This is similar to a strengthening used in Tidor~[\ref{tidor},  Proposition 4.3].
In particular, we do not drop the $\alpha^2 |\mathbb{V}|^2$ term ---corresponding to $\gamma=0$--- in the relation~(\ref{eq:greenst}) below.

\begin{theorem}\label{tri_counting}
Let $r\geq 1$ be an integer,
let $\epsilon>0$,
and let $(E,G)$ be a $\epsilon$-uniform rank-$r$ binary distribution.
Let $\alpha = \frac{|E|}{2^r}$ and let
$T_E$ be the number of ordered triples $(x,y,z)$
of $E$ that form a triangle in $G$. 
Then
$$\left|T_E-\alpha^3 2^{2r}\right| \leq  \epsilon \left(\alpha-\alpha^2\right)2^{2r}.$$
\end{theorem}

\begin{proof}
This proof uses Fourier analysis, which will require us to think of 
the ground set as a subset of a vector space.
Let $\mathbb{V}= \GF(2)^r$  equipped with the standard dot product;
so $\mathbb{V}^* \cong \mathbb{V}$ canonically.
Note $\mathbb{V}\del \{0\}$ is isomorphic to $\PG(r-1,2)$ geometrically 
so we may associate $G$ with $\mathbb{V}\del\{0\}$.
Thus we can also consider $E$ as a subset of $\mathbb{V}$.

We define the Fourier transform of a function $f \colon \mathbb{V} \to \mathbb{C}$ as 
$\widehat{f} \colon \mathbb{V} \to \mathbb{R}$, where 
$\widehat{f}(\gamma)=\sum_{y\in \mathbb{V}} f(y) (-1)^{y\cdot \gamma}$.
These are projections of $f(x)$ into the orthogonal basis
$\{(-1)^{x\cdot \gamma}\}_{\gamma \in \mathbb{V}}$ of $\mathbb{R}^\mathbb{V}$.
Therefore $f(x)=\sum_{\gamma \in \mathbb{V}} \widehat{f}(\gamma) (-1)^{x\cdot \gamma}$.

We define the convolution of two functions  $f,g \colon \mathbb{V} \to \mathbb{C}$ as 
$f*g \colon \mathbb{V} \to \mathbb{R}$, where 
$f*g(x)=\sum_{y\in \mathbb{V}} f(y) g(x-y)$.

We use $1_{E}$ to denote the indicator function
of $E \subseteq \mathbb{V}$.

Note that the ordered triple $(x,y,z)$ in $E$
forms a triangle in $G \cong \mathbb{V}\del \{0\}$ if and only if
$x+y+z=0$ in $\mathbb{V}$.
Thus
\begin{equation*}
\begin{split}
T_E &= \left| \{ (x,y,z) : x,y,z \in E,~x+y+z=0 \} \right| \\
	&=\sum_{\substack{x,y,z \in \mathbb{V}\\
					x+y+z=0}} 
			1_E(x) 1_E(y) 1_E (z)\\
	&=1_E*1_E*1_E(0).
\end{split}
\end{equation*}

Using the convolution theorem (which says that
the Fourier transform of a convolution is the product of the Fourier transforms) 
we have that
$\widehat{(1_E*1_E*1_E)}(\gamma)=\widehat{1_E}(\gamma)^3$
for any $\gamma \in \mathbb{V}$.
Thus
\begin{equation*}
\begin{split}
T_E 	&=1_E*1_E*1_E(0)\\
	&= |\mathbb{V}|^{-1} \sum_{\gamma \in \mathbb{V}} \widehat{(1_E*1_E*1_E)}(\gamma)(-1)^{0\cdot \gamma}\\
	&= |\mathbb{V}|^{-1} \sum_{\gamma \in \mathbb{V}} \widehat{1_E}(\gamma)^3.
\end{split}
\end{equation*}

Note that the term for $\gamma=0$ gives $\alpha^3 |\mathbb{V}|^2$. We now bound the remainder.

For each $\gamma \in \mathbb{V} \del \{0\}$ we have a vector hyperplane $W_\gamma:=\{x\in \mathbb{V}: x\cdot \gamma=0\}$ of $\mathbb{V}$,
which gives us the hyperplane $W_\gamma\del \{0\}$ of $G\cong \mathbb{V}\del \{0\}$.
Thus as $E\subseteq \mathbb{V}\del \{0\}$ is $\epsilon$-uniform,
\begin{equation*}
\begin{split}
\epsilon |\mathbb{V}| = \epsilon 2^r &\geq |~2|E\cap(W_\gamma\del \{0\})|-|E|~| \\
				&= |~|E\cap W_\gamma|-|E\del W_\gamma|~|\\
				&= \left| \widehat{1_E}(\gamma) \right|.
\end{split}
\end{equation*}
So as $\widehat{1_E}(\gamma)^2 \geq 0$ for any $\gamma \in \mathbb{V}$,
\begin{equation*}
\begin{split}
\left|\sum_{\gamma \in \mathbb{V}\del \{0\}} \widehat{1_E}(\gamma)^3 \right|
					&\leq \epsilon |\mathbb{V}| \sum_{\gamma \in \mathbb{V}\del \{0\}} \widehat{1_E}(\gamma)^2.
\end{split}
\end{equation*}
Therefore
$$\left|T_E-\alpha^3 |\mathbb{V}|^2\right| =
|\mathbb{V}|^{-1}\left|\sum_{\gamma \in \mathbb{V}\del \{0\}} \widehat{1_E}(\gamma)^3 \right|
					\leq \epsilon \sum_{\gamma \in \mathbb{V}\del \{0\}} \widehat{1_E}(\gamma)^2.$$

By Parseval's identity (which says that the Fourier transform preserves the $L^2$-norm up to a scalar) 
we have that
\begin{equation}\label{eq:greenst}
\begin{split}
\alpha^2 |\mathbb{V}|^2+ \sum_{\gamma \in V\del \{0\}} \widehat{1_E}(\gamma)^2
		&=|\mathbb{V}| \sum_{x \in \mathbb{V}} 1_E(x)^2\\
		&=\alpha |\mathbb{V}|^2.
\end{split}
\end{equation}

So altogether, we have that
$$\left|T_E-\alpha^3 |\mathbb{V}|^2\right| \leq  \epsilon \left(\alpha-\alpha^2\right)|\mathbb{V}|^2,$$
and thus
$$\left|T_E-\alpha^3 2^{2r}\right| \leq  \epsilon \left(\alpha-\alpha^2\right)2^{2r},$$
as we wanted to show.
\end{proof}

We are now ready to prove our geometric inductive step (Theorem~\ref{geostructind}) for the case $n=3$.
A {\em fano} is a matroid isomorphic to $F_7 \cong \PG(2,2)$.

\begin{theorem}\label{fanofree}
For an integer $r \geq 3$, 
if $(E,G)$ be a fano-free rank-$r$ binary representation
with $|E|> \frac{5}{8}2^{r}$, 
then there is a hyperplane $H$ of $G$ such that $(E\cap H,G|H)$ is triangle-free.
Furthermore, $|E\cap H|> \frac{1}{4}2^{r-1}$.
\end{theorem}
\begin{proof}
Suppose otherwise that each hyperplane $H$ of $G$
contains a triangle of $G$ in $E$.

Let $T_{E}$ be the number of triples $(x,y,z)$ in $E$
that form a triangle in $G$.
This is six times the number of triangles of $G$ in $E$.
Also note that
each triangle in $E$ gives a pair of points in exactly three cones,
thus
$T_{E}=\sum_{p\in E} |E_p|$.
Specifically, we have a bijection given by mapping a triple $(x,y,z)$ that forms a triangle of $G$ in $E$,
to the point $x$ in $E_z$.

\begin{claim}\label{lower}
$T_E > \frac{55}{256} 2^{2r}$.
\end{claim}
\begin{proof}
Let $\alpha=\frac{|E|}{2^{r}}$, which is the density of $E$ considered as a subset of $\mathbb{V}\cong \GF(2)^r$.
Let $\epsilon=\frac{1-\alpha}{3}$.

We first show that $E$ is $\epsilon$-uniform,
then use Theorem~\ref{tri_counting}.

Fix a hyperplane $H$ of $G$.
By contradictory assumption, $E\cap H$ is $\PG(2,2)$-free yet contains a triangle of $G$.
So by Lemma~\ref{Hsize} and Theorem~\ref{binaryBB}, 
$$\frac{1}{2}2^{r-1}<|E\cap H| \leq \frac{3}{4}2^{r-1}.$$

By assumption and Theorem~\ref{binaryBB},
$$\frac{5}{8}2^{r}<|E| \leq \frac{3}{4}2^{r}.$$

Thus as $\epsilon2^r=\frac{1-\alpha}{3}2^r=\frac{2^r-|E|}{3}$,
$$\frac{1}{2}\left(|E|-\epsilon2^r\right) < \frac{1}{2}2^{r-1}<|E\cap H| \leq \frac{3}{4}2^{r-1} < \frac{1}{2}\left(|E|+\epsilon2^r\right).$$
As this holds for any hyperplane $H$ of $G$,
we have that $E$ is $\epsilon$-uniform.

So by Theorem~\ref{tri_counting},

$$\left|T_E-\alpha^3 2^{2r}\right| \leq  \epsilon \left(\alpha-\alpha^2\right)2^{2r}.$$

Therefore
\begin{equation*}
\begin{split}
T_E &\geq \alpha^3 2^{2r}- \epsilon \left(\alpha-\alpha^2\right)2^{2r}\\
	&=\left(\alpha^3-\frac{1-\alpha}{3}(1-\alpha)\alpha\right)2^{2r}\\
	&=\frac{1}{3}\left(2\alpha^3+2\alpha^2-\alpha\right)2^{2r}.
\end{split}
\end{equation*}
So as $\alpha=\frac{|E|}{2^{r}}>\frac{5}{8}$ by assumption on $(E,G)$,
we have that $T_E > \frac{55}{256} 2^{2r}$, as desired.
\end{proof}

\begin{claim}\label{upper}
For each $p\in E$, we have $|E_p| \leq \frac{5}{16} 2^{r}$.
\end{claim}
\begin{proof}
Suppose otherwise that $|E_p|> \frac{5}{16} 2^{r}$.
By Theorem~\ref{cone_ind}, $(E_p,G)$ is a triangle-free rank-$r$ binary representation,
so by Theorem~\ref{GS}, there is a hyperplane $H$ of $G$ that is disjoint from $E_p$.
Note $p \in H$, as otherwise $E_p$ is empty.
Thus all lines of $G$ in $H$ through $p$ can contain at most one other point of $E$,
and so $|E\cap H| \leq 1+|\PG(r-2,2)|=\frac{1}{2} 2^{r-1}$.
This contradicts Lemma~\ref{Hsize} as $H$ contains a triangle of $G$ in $E$ by contradictory assumption.
\end{proof}

By the Goevaerts-Storme Theorem (Theorem~\ref{GS}), $|E| \leq \frac{21}{32}2^{r}$. Hence by Claim~\ref{upper},
$$T_{E}=\sum_{p\in E} |E_p| \leq \frac{21}{32}2^{r} \cdot \frac{5}{16}2^{r}=\frac{105}{512}2^{2r}.$$

Combining this with Claim~\ref{lower}, we have $\frac{105}{512}2^{2r}\geq T_E > \frac{110}{512}2^{2r}$,
a contradiction.

So there must indeed be a hyperplane $H$ of $G$ such that $(E\cap H,G|H)$ is triangle-free,
as we wanted to show.

Furthermore,
$$|E\cap H|=|E|-|E\setminus H|>|E|-|\AG(r-1,2)|> \frac{1}{4}2^{r-1}.$$
\end{proof}

\section{Main Proof}
We proove Theorem~\ref{struct} by induction. Our inductive step is as follows:
\begin{theorem}\label{structind}
For integers $r$ and $n$ with $r\geq n\geq 3$, 
if $M$ is a simple rank-$r$ $\PG(n-1,2)$-free binary matroid with
$|M|> \left(1-\frac{3}{2^{n}}\right)2^{r}$, 
then there is a $\PG(n-2,2)$-free hyperplane $L$ of $M$.
Furthermore, $|L|> \left(1-\frac{3}{2^{n-1}}\right)2^{r-1}$.
\end{theorem}

We first prove the following geometric analogue of our inductive step (Theorem~\ref{structind}).
\begin{theorem}\label{geostructind}
For integers $r$ and $n$ with $r \geq n \geq 2$, 			
if $(E,G)$ is a $\PG(n-1,2)$-free rank-$r$ binary representation
with $|E|> \left(1-\frac{3}{2^{n}}\right)2^{r}$, 
then there is a hyperplane $H$ of $G$ such that $(E\cap H,G|H)$ is $\PG(n-2,2)$-free.
Furthermore, $|E\cap H|> \left(1-\frac{3}{2^{n-1}}\right)2^{r-1}$.
\end{theorem}
\begin{proof}
Consider, for a contradiction, a counterexample $(r,n,(E,G))$ with $n$ minimal.
Thus Theorem~\ref{geostructind} holds for all integers $n'$ with $n> n' \geq 3$,
and by recursive application we have the following.
\begin{remark}\label{geoindstep}								
For integers $n'$ and $r'$ with $r' \geq n' \geq 2$ and $n'<n$, 
if $(E',G')$ is a $\PG(n'-1,2)$-free rank-$r'$ binary representation
with $|E'|> \left(1-\frac{3}{2^{n'}}\right)2^{r'}$, 
then there is a corank-$(n'-2)$ flat $K'$ of $G'$ such that $(E'\cap K',G'|K')$ is triangle-free.
\end{remark}

By Theorem~\ref{fanofree}, we may assume that $n > 3$.

As $(E,G)$ is a counterexample, for each hyperplane $H$ of $G$,
there is a copy of $\PG(n-2,2)$ in $G|(E\cap H)$.
So by Lemma~\ref{Hsize}		
\begin{equation}\label{eq:Hsize}
|E\cap H|>\left(1-\frac{1}{2^{n-2}}\right)2^{r-1},
\end{equation}
for each hyperplane $H$ of $G$.

For any $p \in E$,
consider $(E_p,G)$.
By Lemma~\ref{cone_ind},
$(E_p,G)$ is a $\PG(n-2,2)$-free rank-$r$ binary representation
with $|E_p|> \left(1-\frac{3}{2^{n-1}}\right)2^{r}$.
So by Remark~\ref{geoindstep},  there is a corank-$(n-3)$ flat $K'$ of $G$ such that $(E_p\cap K',G|K')$ is triangle-free.
Note we may assume that $p \in K'$, as otherwise we can take
a different hyperplane of $\cl_G(K'\cup \{p\})$.							
Thus we get that $(E_p\cap K',G|K') = ((E\cap K')_p,G|K')$.
By the Bose-Burton Theorem (Theorem~\ref{binaryBB}), we have
$\left|(E\cap K')_p\right| \leq \frac{1}{2} 2^{r-n+3}$.							
So by Lemma~\ref{cone_ind} we have 
$\left|E\cap K'\right| \leq \frac{3}{4} 2^{r-n+3}$.
As $r_G(K')=r-n+3<r$, we can extend $K'$ to a hyperplane $H$ of $G$.
Note
\begin{equation*}
\begin{split}
|E\cap H| &\leq |H|-|K'|+|E\cap K'| \\
		&\leq 2^{r-1}-2^{r-n+3} +|E\cap K'| \\
		&\leq \left(1-\frac{1}{2^{n-2}}\right)2^{r-1}.
\end{split}
\end{equation*}
This gives a contradiction with the inequality~(\ref{eq:Hsize}).

Thus there is indeed a hyperplane $H$ of $G$ such that $(E\cap H,G|H)$ is $\PG(n-2,2)$-free,
as we wanted to show.

Furthermore,
$$|E\cap H|=|E|-|E\setminus H|>|E|-|\AG(r-1,2)|> \left(1-\frac{3}{2^{n-1}}\right)2^{r-1}.$$
\end{proof}

\subsection{Back to Matroids}\label{back to matroids}

We now need a lemma so that we can get our matroidal inductive step (Theorem~\ref{structind}) from 
our geometric inductive step (Theorem~\ref{geostructind}):
there may be incongruities between a matroid and the ambient space of its geometric representation.
For example, the rank of the matroid may not be the same as the rank of the ambient space.
Additionally, even if the ranks of the matroid and the ambient geometry do agree,
the hyperplanes might not.
That is to say, while any hyperplane of $M=G|E$ spans a hyperplane in $G$,
a hyperplane of $G$ need not intersect $E$ in a hyperplane of $M$.

Lemma~\ref{geomatspecial} shows that in Theorem~\ref{geostructind} we have sufficient conditions for which there are no such incongruities.
As it is easier, we will first prove the following lemma that is already suffiecient for the case $n\geq 4$.

\begin{lemma}\label{geomateasy}
For any integer $r$ with $r\geq 1$,
let $(E,G)$ be a rank-$r$ binary representation 
and let $M=G|E$.
If $|E|\geq \frac{3}{4}2^{r}$, 
then $r(M)=r$ and 
for each hyperplane $H$ of $G$,
the set $E\cap H$ is a hyperplane of $M$.
\end{lemma}
\begin{proof}
Note 
$|E|\geq \frac{3}{4}2^{r}>2^{r-1}-1=|\PG(r-2,2)|,$
so $r(M)=r_G(E)=r$.

For any hyperplane $H$ of $G$,
\begin{equation*}
\begin{split}
|E\setminus H|+|E\cap H|&=|E| \\
					&\geq \frac{3}{4}2^{r} \\
					&>2^{r-1}+2^{r-2}-1 \\
					&=|\AG(r-1,2)|+|\PG(r-3,2)|,
\end{split}
\end{equation*}
so $r_M(E\cap H)=r_G(E\cap H)=r-1$.
\end{proof}

To also include the case $n=3$, 
we need more conditions and more work.

\begin{theorem}\label{geomatspecial}
For any integers $r\geq n \geq 3$, 
let $(E,G)$ be a $\PG(n-1,2)$-free rank-$r$ binary representation
and let $M=G|E$ 
with $|E|> \left(1-\frac{3}{2^{n}}\right)2^{r}$.
Then $M$ is a simple $\PG(n-1,2)$-free rank-$r$ binary matroid,
and for each hyperplane $H$ of $G$,
the set $E\cap H$ is a hyperplane of $M$

\end{theorem}

\begin{proof}
As $|E|> \left(1-\frac{3}{2^{n}}\right)2^{r}> \left(1-\frac{2}{2^{n}}\right)2^{r-1}$,
by the Bose-Burton Theorem (Theorem~\ref{binaryBB}), $r(M)=r$.

Fix a hyperplane $H$ of $G$.
If $G|(E\cap H)$ contains
a copy of $\PG(n-2,2)$, 
then by Lemma~\ref{Hsize}, we have
$$|E\cap H|>\left(1-\frac{2}{2^{n-1}}\right)2^{r-1}>\left(1-\frac{2}{2^{n}}\right)2^{r-2}.$$
So by the Bose-Burton Theorem (Theorem~\ref{binaryBB}), $r(E\cap H)=r-1$.
On the other hand, if $(E\cap H,G|H)$ is $\PG(n-2,2)$-free,
as we have
$$|E \cap H|  > |E| -|\AG(r-1,2)|>\left(1-\frac{3}{2^{n-1}}\right)2^{r-1}>\left(1-\frac{2}{2^{n-1}}\right)2^{r-2},$$
 this implies that
 $r(E\cap H)=r-1$, by the Bose-Burton Theorem (Theorem~\ref{binaryBB}).

In any case, $M=G|E$ is a simple $\PG(n-1,2)$-free rank-$r$ binary matroid
and $E\cap H$ is a hyperplane of $M$
\end{proof}

So for integer $r\geq n \geq 3$,
this allows us to replace 
{\it``a rank-$r$ binary representation $(E,G)$ of $M$ 
and $L:=E\cap H$ for a hyperplane $H$ of $G$"}
with {\it``a simple rank-$r$ binary matroid $M$ with hyperplane $L$"} whenever
$(E,G)$ is $\PG(n-1,2)$-free with $|E|> \left(1-\frac{3}{2^{n}}\right)2^{r}$.
In particular we immediately get Theorem~\ref{structind} from Theorem~\ref{geostructind}.
As this is the inductive step and as the base case holds, we have Theorem~\ref{struct} as well.

%

\section*{References}
\newcounter{refs}
\begin{list}{[\arabic{refs}]}
{\usecounter{refs}\setlength{\leftmargin}{10mm}\setlength{\itemsep}{0mm}}


\item\label{b80}
A. Beutelspacher,
Blocking sets and partial spreads in finite projective spaces,
Geometria Dedicata 9 (1980) 425--449.

\item \label{bb}
R. C. Bose, R. C. Burton,
A characterization of flat spaces in a finite geometry
and the uniqueness of the Hamming and the MacDonald codes,
J. Combin. Theory 1 (1966) 96--104.

\item \label{bt}
S. Brandt, S. Thomass\'{e},
Dense triangle-free graphs are four-colourable:
a solution to the Erd\"{o}s-Simonovits problem, submitted 

\item\label{bw}
A. A. Bruen, D. L. Wehlau,
Long binary linear codes and large caps in projective space,
Designs, Codes and Cryptography 17 (1999), 37-60.

\item\label{dt}
A. A. Davydov, L. M. Tombak,
Quasi perfect linear binary codes with distance 4 and complete caps in projective geometry,
Problems of Information Transmission 25 (1990), no. 4, 265--275


\item\label{gn}
J. Geelen, P. Nelson, 
The critical number of dense triangle-free binary matroids, 
arXiv:1406.2588 [math.CO].


\item\label{gl}
W. Goddard, J. Lyle,
Dense graphs with small clique number,
J. Graph Theory 66 (2011), no. 4, 319--331.

\item\label{gs}
P. Govaerts, L. Storme,
The classification of the smallest nontrivial blocking sets in $\PG(n,2)$,
J. Combin. Theory Ser. A 113 (2006) 1543--1548.

\item\label{green}
B. Green,
A Szemer\'{e}di-type regularity lemma in abelian groups, with applications,
Geom. Fuct. Anal. 15 (2005), 340--376.


\item\label{packingproblem}
J. W. P. Hirschfeld, L. Storme,
The packing problem in statistics, coding theory and finite projective spaces: update 2001,
in Finite Geometries, 
Dev. Math. 3, Kluwer Academic Publishers, Dordrecht, 2001, 201--246.

%



\item\label{tidor}
J. Tidor,
Dense binary $\PG(t-1,2)$-free matroids have critical number $t-1$ or $t$,
arXiv:1508.07278 [math.CO].


\end{list}

\end{document}